\documentclass[10pt]{amsart}

\usepackage{amsmath,amsfonts, latexsym,graphicx, footmisc, amssymb, mathtools, enumerate, mdwlist}

\usepackage{hyperref}
\hypersetup{colorlinks=true, urlcolor=blue, citecolor=blue, linkcolor=blue}

\newcommand{\U}{{\mathcal U}}

\newcommand{\C}{{\mathbb C}}
\newcommand{\Z}{{\mathbb Z}}

\newcommand{\Q}{{\mathbb Q}}

\newcommand{\hyp}{{\mathbb H}}
\newcommand{\supp}{\operatorname{supp}}
\newcommand{\cosupp}{\operatorname{cosupp}}

\newcommand{\im}{\mathop{\rm im}\nolimits}

\newcommand{\arrow}[1]{\stackrel{#1}{\longrightarrow}}
\newcommand{\Adot}{\mathbf A^\bullet}

\newcommand{\Idot}{\mathbf I^\bullet}
\newcommand{\Jdot}{\mathbf J^\bullet}
   
\newcommand{\Fdot}{\mathbf F^\bullet}   

\newcommand{\Pdot}{\mathbf P^\bullet}

\newcommand{\vdual}{{\mathcal D}}

\newcommand{\coker}{{\operatorname{coker}}}
\newcommand{\id}{{\operatorname{id}}}

\newtheorem{defn0}{Definition}[section]
\newtheorem{prop0}[defn0]{Proposition}
\newtheorem{conj0}[defn0]{Conjecture}
\newtheorem{thm0}[defn0]{Theorem}
\newtheorem{lem0}[defn0]{Lemma}
\newtheorem{corollary0}[defn0]{Corollary}
\newtheorem{example0}[defn0]{Example}
\newtheorem{remark0}[defn0]{Remark}
\newtheorem{question0}[defn0]{Question}
\newtheorem{exercise0}[defn0]{Exercise}

\newenvironment{defn}{\begin{defn0}}{\end{defn0}}
\newenvironment{prop}{\begin{prop0}}{\end{prop0}}

\newenvironment{thm}{\begin{thm0}}{\end{thm0}}
\newenvironment{lem}{\begin{lem0}}{\end{lem0}}

\newenvironment{rem}{\begin{remark0}\rm}{\end{remark0}}
\newenvironment{ques}{\begin{question0}\rm}{\end{question0}}

\newcommand{\defref}[1]{Definition~\ref{#1}}
\newcommand{\propref}[1]{Proposition~\ref{#1}}
\newcommand{\thmref}[1]{Theorem~\ref{#1}}
\newcommand{\lemref}[1]{Lemma~\ref{#1}}

\newcommand{\secref}[1]{Section~\ref{#1}}

\title[Intermediate Extension, Vanishing Cycles, and Eigenspace of One]{The Intermediate Extension, Vanishing Cycles, and  Perverse Eigenspace of One}


\author{David B. Massey}

\date{}

\begin{document}

\begin{abstract} We prove a number of  results involving the kernel of the identity minus the monodromy on the vanishing cycles.
\end{abstract}

\maketitle

%\newpage

%\tableofcontents

%\newpage

\section {Introduction}

Suppose that $\U$ is a non-empty open subset of $\C^{n+1}$ and $f:\U\rightarrow\C$ is a nowhere locally constant, reduced, complex analytic function such that $f^{-1}(0)$ is non-empty. Then $V(f)=f^{-1}(0)$ is an affine hypersurface inside $\U$. 

In Theorem 2.5 of  \cite{perveigen1}, we proved that we have a short exact sequence in the abelian category of perverse sheaves on $V(f)$ (using $\Z$ as our base ring):
$$
0\to \ker\{\operatorname{id}-\widetilde T_f\}\to \Z^\bullet_{V(f)}[n]\rightarrow \Idot_{V(f)}\to 0,\leqno{(\dagger)}
$$

\noindent where $\Idot_{V(f)}$ is the intersection cohomology complex on $V(f)$ (with constant $\Z$ coefficients) and $\widetilde T_f$ is the vanishing cycle monodromy operator.

\medskip

In Theorem 2.6 of \cite{perveigen1}, we  note that the dual statement is also true. That dual result is, letting $j:V(f)\rightarrow\U$ denote the inclusion,  there exists a short exact sequence
$$
0\to \Idot_{V(f)}\to j^![1]\Z_\U^\bullet[n+1]\to\coker\{\operatorname{id}-\widetilde T_f\}\to 0,\leqno{(\ddagger)}
$$
in $\operatorname{Perv}(V(f))$.

\medskip

In this paper, we will prove a number of other results involving $\ker\{\operatorname{id}-\widetilde T_f\}$, which is the perverse eigenspace of one for the vanishing cycle monodromy; in particular, in \thmref{thm:mainshort}, we prove generalizations of ($\dagger$) and ($\ddagger$).

\medskip

\section{Preliminary notation, definitions and results}\label{sec:basic}

In this section, we give notations and well-known results on the derived category and perverse sheaves. Basic references are \cite{bbd}, \cite{inthom2}, \cite{boreletal}, \cite{iversen}, \cite{kashsch}, \cite{schurbook}, \cite{dimcasheaves}, and \cite{MaximIH}. We should remark that we adopt a convention which is now standard; because we will work solely in the derived category, we will {\bf not} write an $R$ in front of derived functors since all of our functors are derived.

\medskip

Let $X$ be a complex analytic space. We use $D^b_c(X)$ to denote the derived category of bounded, constructible complexes of sheaves of $\Z$-modules on $X$. We let ${}^{{}^\mu}\mathbf D^{{}^{\leqslant 0}}(X)$ (respectively, ${}^{{}^\mu}\mathbf D^{{}^{\geqslant 0}}(X)$) denote the full subcategory of $D^b_c(X)$ of those complexes which satisfy the support (respectively, cosupport) condition. We note the abelian category of perverse sheaves on $X$ by $\operatorname{Perv}(X)$; thus $\operatorname{Perv}(X)={}^{{}^\mu}\mathbf D^{{}^{\leqslant 0}}(X)\cap {}^{{}^\mu}\mathbf D^{{}^{\geqslant 0}}(X)$. We let ${}^{\mu}\hskip -0.02in H^k$ denote the degree $k$ cohomology with respect to the perverse $t$-structure.
\medskip

Let $f:X\rightarrow\C$ be a nowhere locally constant complex analytic function on $X$. We let $j: V(f)\hookrightarrow X$ be the (closed) inclusion and $i: X\backslash V(f)\hookrightarrow X$ be the (open) inclusion. Let $\psi_f[-1]$ and $\phi_f[-1]$ denote the shifted nearby cycle and vanishing cycle functors, respectively, from $D^b_c(X)$ to $D^b_c(V(f))$; we denote the respective Milnor monodromy (natural) automorphisms of these functors by $T_f$ and $\widetilde T_f$.

The functors $\psi_f[-1]$ and $\phi_f[-1]$ are $t$-exact with respect to the perverse $t$-structure. This means that each of these functors takes ${}^{{}^\mu}\mathbf D^{{}^{\leqslant 0}}(X)$ (respectively, ${}^{{}^\mu}\mathbf D^{{}^{\geqslant 0}}(X)$) to ${}^{{}^\mu}\mathbf D^{{}^{\leqslant 0}}(V(f))$ (respectively, ${}^{{}^\mu}\mathbf D^{{}^{\geqslant 0}}(V(f))$). In particular, 
this means that $\psi_f[-1]$ and $\phi_f[-1]$ take perverse sheaves to perverse sheaves.

\bigskip

For $\Adot\in D^b_c(X)$, there are natural distinguished triangles:

\bigskip

\noindent {\bf The vanishing triangle}:
$$
\rightarrow j^*[-1]\Adot\xrightarrow{\operatorname{comp}}\psi_f[-1]\Adot\xrightarrow{\operatorname{can}}\phi_f[-1]\Adot\arrow{[1]};
$$

\medskip

and

\bigskip

\noindent {\bf The dual vanishing triangle}:
$$
\rightarrow \phi_f[-1]\Adot\xrightarrow{\operatorname{var}}\psi_f[-1]\Adot\xrightarrow{\operatorname{pmoc}} j^![1]\Adot\arrow{[1]}.
$$

\bigskip

Automorphisms of these triangles are given respectively by $(\operatorname{id}, T_f, \widetilde T_f)$ and $(\widetilde T_f, T_f, \operatorname{id})$. Furthermore, $\operatorname{var}\circ\operatorname{can}=\operatorname{id}-T_f$ and $\operatorname{can}\circ\operatorname{var}=\operatorname{id}-\widetilde T_f$.

\bigskip

It is very helpful to give a name to one more map:

\begin{defn}
We let $\omega_f:= \operatorname{pmoc}\circ\operatorname{comp}$ be the natural transformation from $j^*[-1]$ to  $j^![1]$ and refer to this as the {\bf Wang transformation} (or the {\bf Wang morphism} for a given $\Adot$).
\end{defn}

\bigskip

Let $d$ be an integer, and let $f:Y\rightarrow X$ be a morphism of complex spaces such that, for all $\bold x\in X$, 
$\operatorname{dim}f^{-1}(\bold x)\leqslant d$. Then,

\vskip .1in

\noindent 1) $f^*$ sends ${}^{{}^\mu}\bold D^{{}^{\leqslant 0}}(X)$ to ${}^{{}^\mu}\bold D^{{}^{\leqslant d}}(Y)$;

\vskip .1in

\noindent 2) $f^!$ sends ${}^{{}^\mu}\bold D^{{}^{\geqslant 0}}(X)$ to ${}^{{}^\mu}\bold D^{{}^{\geqslant -d}}(Y)$;

\vskip .1in

\noindent 3) if $\Fdot\in {}^{{}^\mu}\bold D^{{}^{\leqslant 0}}(Y)$ and  $f_!\Fdot\in \bold D^b_c(X)$, then $f_!\Fdot\in
{}^{{}^\mu}\bold D^{{}^{\leqslant d}}(X)$;

\vskip .1in

\noindent 4) if $\Fdot\in {}^{{}^\mu}\bold D^{{}^{\geqslant 0}}(Y)$ and  $f_*\Fdot\in \bold D^b_c(X)$, then $f_*\Fdot\in
{}^{{}^\mu}\bold D^{{}^{\geqslant -d}}(X)$.

\vskip 0.3in

For closed embeddings, we also have the following:

\vskip 0.1in

\noindent Let $g_1, \dots, g_e$ be complex analytic functions on $X$. Let $m$ denote the inclusion of $Y:=V(g_1, \dots, g_e)$ into $X$. Then,

\vskip 0.1in

\noindent i) $m^*$ sends ${}^{{}^\mu}\bold D^{{}^{\geqslant 0}}(X)$ to ${}^{{}^\mu}\bold D^{{}^{\geqslant -e}}(Y)$;

\vskip 0.1in

\noindent ii) $m^!$ sends ${}^{{}^\mu}\bold D^{{}^{\leqslant 0}}(X)$ to ${}^{{}^\mu}\bold D^{{}^{\leqslant e}}(Y)$.

\medskip

\section{A Fundamental Exact Sequence}

Suppose that $\Pdot\in\operatorname{Perv}(X)$. Then, applying perverse cohomology to the vanishing and dual vanishing triangles, we immediately conclude:

\begin{prop}\label{prop:prewang} ${}^{\mu}\hskip -0.02in H^k(j^![1]\Pdot)$ is possibly non-zero only for $k=-1, 0$, ${}^{\mu}\hskip -0.02in H^k(j^*[-1]\Pdot)$ is possibly non-zero only for $k=0, 1$, and there are exact sequences

\smallskip

$$
0\rightarrow {}^{\mu}\hskip -0.02in H^0(j^*[-1]\Pdot)\xrightarrow{{}^{\mu}\hskip -0.02in H^0(\operatorname{comp})}\psi_f[-1]\Pdot\xrightarrow{\operatorname{can}}\phi_f[-1]\Pdot\rightarrow {}^{\mu}\hskip -0.02in H^1(j^*[-1]\Pdot)\rightarrow 0;
$$

\medskip

and
$$
0\rightarrow {}^{\mu}\hskip -0.02in H^{-1}(j^![1]\Pdot)\rightarrow \phi_f[-1]\Pdot\xrightarrow{\operatorname{var}}\psi_f[-1]\Pdot\xrightarrow{{}^{\mu}\hskip -0.02in H^0(\operatorname{pmoc})}{}^{\mu}\hskip -0.02in H^0(j^![1]\Pdot)\rightarrow 0.
$$

\medskip

Thus, ${}^{\mu}\hskip -0.02in H^{0}(j^*[-1]\Pdot)\cong\ker\{\operatorname{can}\}$, ${}^{\mu}\hskip -0.02in H^{1}(j^*[-1]\Pdot)\cong\coker\{\operatorname{can}\}$, ${}^{\mu}\hskip -0.02in H^{-1}(j^![1]\Pdot)\cong\ker\{\operatorname{var}\}$, and ${}^{\mu}\hskip -0.02in H^{0}(j^![1]\Pdot)\cong\coker\{\operatorname{var}\}$.

\end{prop}

\medskip

\begin{thm}\label{thm:wang} Suppose that $\Pdot$ is a perverse sheaf on $X$. Then, there is an exact sequence in $\operatorname{Perv}(V(f))$:

$$
0\rightarrow\frac{\ker\{\id-\widetilde T_f\}}{\ker\{\operatorname{var}\}}\longrightarrow{}^{\mu}\hskip -0.02in H^0(j^*[-1]\Pdot)\xlongrightarrow{ {}^{\mu}\hskip -0.02in H^0(\omega_f)}{}^{\mu}\hskip -0.02in H^0(j^![1]\Pdot)\longrightarrow\frac{\operatorname{im}\{\operatorname{can}\}}{\operatorname{im}\{\id-\widetilde T_f\}}\rightarrow 0,
$$

\medskip

\noindent where we view $\ker\{\operatorname{var}\}$ as a sub-perverse sheaf of $\ker\{\id-\widetilde T_f\}$ by the canonical injection $\ker\{\operatorname{var}\}\hookrightarrow\ker\{\operatorname{can}\circ\operatorname{var}\}=\ker\{\id-\widetilde T_f\}$, and
 we view $\im\{\id-\widetilde T_f\}$ as a sub-perverse sheaf of $\im\{\operatorname{can}\}$ by the canonical injection $\im\{\operatorname{\id-\widetilde T_f}\}=\im\{\operatorname{can}\circ\operatorname{var}\}\hookrightarrow\im\{\operatorname{can}\}$.
\end{thm}

\begin{proof} One easily verifies that there is an exact sequence

\medskip

$0\rightarrow\ker\{\operatorname{var}\}\rightarrow \ker\{\operatorname{can}\circ\operatorname{var}\}\rightarrow \ker\{\operatorname{can}\}\rightarrow \hfill$

\medskip

$\hfill\coker\{\operatorname{var}\}\rightarrow  \coker\{\operatorname{can}\circ\operatorname{var}\}\rightarrow \coker\{\operatorname{can}\}\rightarrow 0,
$

\smallskip

\noindent where

\begin{itemize}
\item the second arrow from the left is the canonical injection, 
\medskip
\item the third arrow is induced by $\operatorname{var}$,
\medskip
\item the fourth arrow is the composition of the canonical injection of $\ker\{\operatorname{can}\}$ into $\psi_f[-1]\Pdot$ with the canonical surjection from $\psi_f[-1]\Pdot$ onto $\coker\{\operatorname{var}\}$,
\medskip
\item the fifth arrow is induced by $\operatorname{can}$, and
\medskip
\item the sixth arrow is the canonical surjection.
\end{itemize}

\bigskip

Now the exact sequence in the statement of the theorem follows immediately from  \propref{prop:prewang}.
\end{proof}

\medskip

\section{Splitting}

\medskip

We have the following easy result which tells us when the vanishing cycles are a direct summand of the nearby cycles.

\smallskip

\begin{thm}\label{thm:split} Let $\Pdot$ be a perverse sheaf on $X$. Then, the following are equivalent:

\begin{enumerate}
\item $\id-\widetilde T_f$ is an isomorphism;
\medskip
\item $\ker\{\id-\widetilde T_f\}=0$ and $\coker\{\id-\widetilde T_f\}=0$;
\medskip
\item ${}^{\mu}\hskip -0.02in H^{-1}(j^![1]\Pdot)=0$, ${}^{\mu}\hskip -0.02in H^{1}(j^*[-1]\Pdot)=0$, and ${}^{\mu}\hskip -0.02in H^0(\omega_f)$ is an isomorphism;
\medskip
\item $j^![1]\Pdot$ and $j^*[-1]\Pdot$ are perverse, and $\omega_f: j^*[-1]\Pdot\rightarrow j^![1]\Pdot$ is an isomorphism;
\medskip
\item $\omega_f: j^*[-1]\Pdot\rightarrow j^![1]\Pdot$ is an isomorphism, and $\Pdot$ is the intermediate extension of $i^*\Pdot=i^!\Pdot$ to $V(f)$.
\end{enumerate}

\medskip

Furthermore, when these equivalent conditions hold, the vanishing and dual vanishing triangles are short exact sequences in $\operatorname{Perv}(V(f))$ which split in a manner compatible with the monodromy automorphisms; thus we have isomorphisms
$$
\psi_f[-1]\Pdot\cong \phi_f[-1]\Pdot\oplus j^*[-1]\Pdot\cong \phi_f[-1]\Pdot\oplus j^![1]\Pdot
$$
and, via these isomorphisms, $T_f$ is identified with $(\widetilde T_f, \id)$ in each case.
\end{thm}
\begin{proof} Given \propref{prop:prewang} and \thmref{thm:wang}, the equivalences of (1), (2), and (3) are utterly trivial. Clearly (4) implies (3), and (3)  together with  \thmref{thm:wang} implies (4), since having zero perverse cohomology outside of degree 0 is equivalent to being perverse (see 1.3.7 of \cite{bbd} or Proposition 5.1.7 of \cite{dimcasheaves}).

We need to show the equivalence of (4) and (5). Recall the standard results (i) and (ii) from \secref{sec:basic}, which tell us that $j^*[-1]$ sends ${}^{{}^\mu}\bold D^{{}^{\geqslant 0}}(X)$ to ${}^{{}^\mu}\bold D^{{}^{\geqslant 0}}(V(f))$ and  $j^![1]$ sends ${}^{{}^\mu}\bold D^{{}^{\leqslant 0}}(X)$ to ${}^{{}^\mu}\bold D^{{}^{\leqslant 0}}(V(f))$; therefore, $j^*[-1]\Pdot\in {}^{{}^\mu}\bold D^{{}^{\geqslant 0}}(V(f))$ and $j^![1]\Pdot\in {}^{{}^\mu}\bold D^{{}^{\leqslant 0}}(V(f))$. But one of the equivalent definitions/characterizations of $\Pdot$ being the intermediate extension of $i^*\Pdot=i^!\Pdot$ is that $j^*[-1]\Pdot\in {}^{{}^\mu}\bold D^{{}^{\leqslant 0}}(V(f))$ and $j^![1]\Pdot\in {}^{{}^\mu}\bold D^{{}^{\geqslant 0}}(V(f))$ (see, for instance, \cite{dimcasheaves} Definition 5.2.6). Thus (4) and (5) are equivalent.

We need to show that (1)-(5) imply the splittings exist. Assuming (1)-(5), we have short exact sequences in $\operatorname{Perv}(X)$:
$$
0\rightarrow j^*[-1]\Pdot\xrightarrow{\operatorname{comp}}\psi_f[-1]\Pdot\xrightarrow{\operatorname{can}}\phi_f[-1]\Pdot\rightarrow 0
$$
and
$$
0\rightarrow \phi_f[-1]\Pdot\xrightarrow{\operatorname{var}}\psi_f[-1]\Pdot\xrightarrow{\operatorname{pmoc}} j^![1]\Pdot\rightarrow 0,
$$

\medskip

\noindent where $\operatorname{can}\circ\operatorname{var}=\operatorname{id}-\widetilde T_f$.

Consider $\operatorname{var}\circ (\operatorname{id}-\widetilde T_f)^{-1}$ from $\phi_f[-1]\Pdot$ to $\psi_f[-1]\Pdot$, and  $(\operatorname{id}-\widetilde T_f)^{-1}\circ\operatorname{can}$ from $\psi_f[-1]\Pdot$ to $\phi_f[-1]\Pdot$. Then $\operatorname{can}\circ[\operatorname{var}\circ (\operatorname{id}-\widetilde T_f)^{-1}]=\operatorname{id}$ and $[(\operatorname{id}-\widetilde T_f)^{-1}\circ\operatorname{can}]\circ \operatorname{var}=\operatorname{id}$. The first equality shows that the first short exact sequence splits, and the second equality shows that the second short exact sequence splits.
\end{proof}

\begin{rem} Throughout this paper, we use $\Z$ as our base ring because we care about torsion. However, we could use any base ring, $R$, which is a commutative, regular, Noetherian ring, with finite Krull dimension (e.g., $\Z$, $\Q$, or $\C$). In particular, if we use a base ring which is a field in \thmref{thm:split}, then $\ker\{\id-\widetilde T_f\}=0$ if an only if $\coker\{\id-\widetilde T_f\}=0$; so, in the field case, $\ker\{\id-\widetilde T_f\}=0$ if and only if $\id-\widetilde T_f$ is an isomorphism.
\end{rem}

\medskip

\section{The intermediate extension to $V(f)$}

In this section, we isolate the properties that allowed us to prove ($\dagger$) and ($\ddagger$) from the introduction; this allows us to obtain analogous results in a much more general setting.

\smallskip

We continue with $X$ $f$, $j$ and $i$ as before, and assume throughout the remainder of the paper that $\Pdot$ is a perverse sheaf on $X$. We let $\Sigma_f:= \operatorname{supp}\phi_f[-1]\Pdot$, $m:\Sigma_f\hookrightarrow V(f)$ denote the (closed) inclusion and let $\ell:V(f)\backslash\Sigma_f\hookrightarrow V(f)$ denote the (open) inclusion. Finally, we define the (closed) inclusion $\hat m:=j\circ m:\Sigma_f\rightarrow X$. 
\bigskip

First, we have the easy:

\medskip

\begin{prop}\label{prop:pred} The following are equivalent:
\begin{enumerate}

\item ${}^{\mu}\hskip -0.02in H^{-1}(\hat m^*\Pdot)={}^{\mu}\hskip -0.02in H^{0}(\hat m^*\Pdot)=
{}^{\mu}\hskip -0.02in H^{0}(\hat m^!\Pdot)={}^{\mu}\hskip -0.02in H^{1}(\hat m^!\Pdot)=0$;
\medskip
\item $\hat m^*\Pdot\in {}^{{}^\mu}\mathbf D^{{}^{\leqslant {-2}}}(\Sigma_f)$ and $\hat m^!\Pdot\in {}^{{}^\mu}\mathbf D^{{}^{\geqslant {2}}}(\Sigma_f)$;
\medskip
\item $\hat m^*[-2]\Pdot\in {}^{{}^\mu}\mathbf D^{{}^{\leqslant {0}}}(\Sigma_f)$ and $\hat m^![2]\Pdot\in {}^{{}^\mu}\mathbf D^{{}^{\geqslant {0}}}(\Sigma_f)$;
\medskip
\item for all integers $k$,
$$
\dim\supp^{-k}(\hat m^*[-2]\Pdot)=\dim \overline{\big\{x\in\Sigma_f\, |\, H^{-k-2}(\Pdot)_x\neq 0\big\}}\leq k
$$
and
$$
\dim\cosupp^{k}(\hat m^![2]\Pdot)=\dim \overline{\big\{x\in\Sigma_f\, |\, \hyp^{k+2}(B^\circ_\epsilon(x)\cap X,\, B^\circ_\epsilon(x)\cap X\backslash\{x\};\, \Pdot)\neq 0\big\}}\leq k
$$

\medskip

\noindent where $B^\circ_\epsilon(x)$ denotes a open ball of (small) radius $\epsilon>0$, centered at $x$ and, by convention, the empty set has dimension $-\infty$
\end{enumerate}
\end{prop}
\begin{proof} Since $\hat m$ is inclusion, $\hat m^*$ sends ${}^{{}^\mu}\mathbf D^{{}^{\leqslant {0}}}(X)$ to ${}^{{}^\mu}\mathbf D^{{}^{\leqslant {0}}}(\Sigma_f)$, and $\hat m^!$ sends  ${}^{{}^\mu}\mathbf D^{{}^{\geqslant {0}}}(X)$ to ${}^{{}^\mu}\mathbf D^{{}^{\geqslant {0}}}(\Sigma_f)$. Thus Items (1) and (2) are equivalent, and clearly Item (3) is equivalent to Item (2).

Item (4) is nothing more than a direct translation of the conditions (the support and cosupport conditions) in Item (3). To see this, for all $x\in\Sigma_f$, let $w_x:\{x\}\hookrightarrow \Sigma f$ so that $\check w_x:=\hat m\circ w_x$ is the inclusion of $\{x\}$ into $X$. Then use that $H^{-k-2}(\check w_x^*\Pdot)\cong H^{-k-2}(\Pdot)_x$ and $H^{k+2}(\check w_x^!\Pdot)\cong \hyp^{k+2}(B^\circ_\epsilon(x)\cap X,\, B^\circ_\epsilon(x)\cap X\backslash\{x\};\, \Pdot)$.
\end{proof}

\medskip

\begin{rem} The conditions $\hat m^*\Pdot\in {}^{{}^\mu}\mathbf D^{{}^{\leqslant {-2}}}(\Sigma_f)$ and $\hat m^!\Pdot\in {}^{{}^\mu}\mathbf D^{{}^{\geqslant {2}}}(\Sigma_f)$ are dual to each other, provided that we use a field for our base ring. To be precise, if our base ring is a field, and $\Pdot$ is self-dual (i.e., $\vdual\Pdot\cong\Pdot$), then 
$\hat m^*\Pdot\in {}^{{}^\mu}\mathbf D^{{}^{\leqslant {-2}}}(\Sigma_f)$ if and only if $\hat m^!\Pdot\in {}^{{}^\mu}\mathbf D^{{}^{\geqslant {2}}}(\Sigma_f)$.
\end{rem}

\bigskip

\begin{defn}\label{def:pred} When the equivalent conditions of \propref{prop:pred} are satisfied, we say that $f$ is {\bf $\Pdot$-reduced}.
\end{defn}

\smallskip

The reason for this terminology wll become clear in the next section.

\bigskip

In order to prove a generalization of ($\dagger$) and ($\ddagger$), we need something that replaces intersection cohomology of $V(f)$ with constant coefficients. That ``something'' is:

\begin{defn}\label{defn:intermediate}
 Let $\Idot_{V(f)}$ be the intermediate extension from $V(f)\backslash\Sigma_f$ to $V(f)$ of the perverse sheaf
$$\ell^*j^*[-1]\Pdot\cong \ell^*\psi_f[-1]\Pdot\cong\ell^*j^![1]\Pdot\cong \ell^!j^![1]\Pdot\cong\ell^!\psi_f[-1]\Pdot\cong\ell^!j^*[-1]\Pdot.
$$
(Note that these isomorphisms follow at once from the vanishing and dual vanishing triangles, since we are restricting to the complement on the support of the vanishing cycles.)
\end{defn}

\medskip

Now we have:

\medskip

\begin{thm}\label{thm:mainshort} Suppose that $f$ is $\Pdot$-reduced.

\medskip

Then, 

\begin{enumerate}

\item $j^*[-1]\Pdot$ and $j^![1]\Pdot$ are perverse, 

\medskip

\item $\Pdot$ is isomorphic to the intermediate extension of $i^*\Pdot=i^!\Pdot$ from $X\backslash V(f)$ to $X$, 

\medskip

\item $V(f)$ contains no irreducible component of $\supp \Pdot$, 

\medskip

\item $\Sigma_f$ contains no irreducible component of $\supp j^*[-1]\Pdot$ or $\supp j^![1]\Pdot$, and

\medskip

\item there are short exact sequences in $\operatorname{Perv}(V(f))$:
$$
0\rightarrow \ker\{\operatorname{id}-\widetilde T_f\}\rightarrow j^*[-1]\Pdot\rightarrow \Idot_{V(f)}\rightarrow 0
$$
and
$$
0\rightarrow \Idot_{V(f)} \rightarrow j^![1]\Pdot\rightarrow \coker\{\operatorname{id}-\widetilde T_f\}\rightarrow  0
$$
\end{enumerate}
\end{thm}
\begin{proof} 

\phantom{filler}

\medskip

\noindent Proof of (1): Recall that we showed in \thmref{thm:wang} that ${}^{\mu}\hskip -0.02in H^k(j^![1]\Pdot)$ is possibly non-zero only for $k=-1, 0$, and ${}^{\mu}\hskip -0.02in H^k(j^*[-1]\Pdot)$ is possibly non-zero only for $k=0, 1$. Thus, to show (1), we need to show that ${}^{\mu}\hskip -0.02in H^{-1}(j^![1]\Pdot)=0$ and ${}^{\mu}\hskip -0.02in H^{1}(j^*[-1]\Pdot)=0$. We will show that ${}^{\mu}\hskip -0.02in H^{1}(j^*[-1]\Pdot)=0$ and leave the dual argument for ${}^{\mu}\hskip -0.02in H^{-1}(j^![1]\Pdot)$ to the reader. 

Note that $\ell^!j^*[-1]\Pdot$ is perverse and so $\ell_!\ell^!j^*[-1]\Pdot$ satisfies the support condition, i.e.,  ${}^{\mu}\hskip -0.02in H^k(\ell_!\ell^!j^*[-1]\Pdot)=0$ for $k\geq 1$. Also note that our hypothesis that ${}^{\mu}\hskip -0.02in H^{0}(\hat m^*\Pdot)=0$ implies that
$${}^{\mu}\hskip -0.02in H^{1}(m_*m^*j^*[-1]\Pdot)={}^{\mu}\hskip -0.02in H^{0}(m_*\hat m^*\Pdot)=0.$$
Now apply perverse cohomology to the distinguished triangle
$$
\rightarrow \ell_!\ell^!j^*[-1]\Pdot\rightarrow j^*[-1]\Pdot\rightarrow m_*m^*j^*[-1]\Pdot\xrightarrow{[1]}
$$
to reach the desired conclusion.

\bigskip

\noindent Proof of (2): That $\Pdot$ is isomorphic to the intermediate extension of $i^*\Pdot=i^!\Pdot$ from $X\backslash V(f)$ to $X$ follows immediately from (1), since one of the equivalent characterizations of the intermediate extension is that $j^*[-1]\Pdot\in {}^{{}^\mu}\mathbf D^{{}^{\leqslant {0}}}(V(f))$ and  $j^![1]\Pdot\in {}^{{}^\mu}\mathbf D^{{}^{\geqslant {0}}}(V(f))$. See, for instance, \cite{dimcasheaves}, Definition 5.2.6.

\bigskip

\noindent Proof of (3): This follows immediately from (2).

\bigskip

\noindent Proof of (4): Suppose that $\Sigma_f$ contains an irreducible component $C$ of $\supp j^*[-1]\Pdot$ or $\supp j^![1]\Pdot$, where we know that $\supp j^*[-1]\Pdot$ and $\supp j^![1]\Pdot$ are perverse by (1). Then, restricting to an open dense subset of $C$ either $m^*j^*[-1]\Pdot$ or $m^!j^![1]\Pdot$ would be perverse and non-zero. But this would contradict either ${}^{\mu}\hskip -0.02in H^{-1}(\hat m^*\Pdot)=0$ or ${}^{\mu}\hskip -0.02in H^{1}(\hat m^!\Pdot)=0$.

\bigskip

\noindent Proof of (5): By (1), \propref{prop:prewang}, and \thmref{thm:wang}, we have an exact sequence in $\operatorname{Perv}(V(f))$:

$$0\rightarrow \ker\{\id-\widetilde T_f\}\rightarrow j^*[-1]\Pdot\xlongrightarrow{ \omega_f} j^![1]\Pdot\rightarrow  \coker\{\id-\widetilde T_f\}\rightarrow 0.$$
Let $\Jdot:=\im\{\omega_f\}$. Then we have two short exact sequences in $\operatorname{Perv}(V(f))$:

$$0\rightarrow \ker\{\id-\widetilde T_f\}\rightarrow j^*[-1]\Pdot\rightarrow\Jdot\rightarrow 0.$$
and
$$0\rightarrow \Jdot\rightarrow j^![1]\Pdot\rightarrow  \coker\{\id-\widetilde T_f\}\rightarrow 0.$$

\medskip

\noindent We claim that $\Jdot\cong \Idot_{V(f)}$. 

\medskip

First, by applying $\ell^*=\ell^!$ to the short exact sequences, we see that $\Jdot$ is an extension of $\ell^*j^*[-1]\Pdot\cong \ell^!j^![1]\Pdot$. We need to show that $m^*\Jdot\in {}^{{}^\mu}\mathbf D^{{}^{\leqslant {-1}}}(\Sigma_f)$ and $m^!\Jdot\in {}^{{}^\mu}\mathbf D^{{}^{\geqslant {1}}}(\Sigma_f)$, i.e., that ${}^{\mu}\hskip -0.02in H^{0}(\hat m^*\Jdot)=0$ and ${}^{\mu}\hskip -0.02in H^{0}(\hat m^!\Jdot)=0$.

Apply $m^*$ to the first short exact sequence/distinguished triangle above, apply $m!$ to the second short exact sequence/distinguished triangle, and take the long exact sequences on perverse cohomology. We obtain exact sequences
$$
\rightarrow  {}^{\mu}\hskip -0.02in H^{0}(m^* j^*[-1]\Pdot)\rightarrow{}^{\mu}\hskip -0.02in H^{0}(m^*\Jdot)\rightarrow {}^{\mu}\hskip -0.02in H^{1}(m^*\ker\{\id-\widetilde T_f\})\rightarrow
$$
and
$$
\rightarrow {}^{\mu}\hskip -0.02in H^{-1}(m^!\coker\{\id-\widetilde T_f\})\rightarrow {}^{\mu}\hskip -0.02in H^{0}(m^!\Jdot)\rightarrow {}^{\mu}\hskip -0.02in H^{0}(m^! j^![1]\Pdot)\rightarrow.
$$

\medskip

\noindent Now ${}^{\mu}\hskip -0.02in H^{1}(m^*\ker\{\id-\widetilde T_f\})=0$ and ${}^{\mu}\hskip -0.02in H^{-1}(m^!\coker\{\id-\widetilde T_f\})=0$ because $m^*\ker\{\id-\widetilde T_f\}$ and $m^!\coker\{\id-\widetilde T_f\}$ are perverse since they are restrictions/upper-shrieks of perverse sheaves to sets containing the supports of the initial perverse sheaves. Furthermore, 
${}^{\mu}\hskip -0.02in H^{0}(m^* j^*[-1]\Pdot)={}^{\mu}\hskip -0.02in H^{-1}(\hat m^*\Pdot)=0$ and ${}^{\mu}\hskip -0.02in H^{0}(m^! j^![1]\Pdot)= {}^{\mu}\hskip -0.02in H^{1}(\hat m^!\Pdot)=0$ by hypothesis. And so, we are finished.
\end{proof}

\medskip

\section{The constant sheaf on affine space}

Let us consider the classical case from the introduction, where we replace $X$ with $\U$, a non-empty open subset of $\C^{n+1}$, and replace $\Pdot$ with $\Z^\bullet_\U[n+1]$.

\bigskip

We begin with an easy, but fundamental, lemma. This lemma is well-known, but the proof is short, so we include it.

\begin{lem}\label{lem:fund} Let $Y$ be a closed complex analytic subspace of $\U$, which we do not assume is pure-dimensional. Denote by $c$ the codimension of $Y$ in $\U$, i.e., let $c:=n+1-\dim Y$. Let $r: Y\hookrightarrow\U$ be the inclusion.

Then, 
$$r^*[-c]\Z^\bullet_\U[n+1]\in {}^{{}^\mu}\mathbf D^{{}^{\leqslant {0}}}(Y)\hskip 0.2in\textnormal{ and }\hskip 0.2in r^![c]\Z^\bullet_\U[n+1]\in {}^{{}^\mu}\mathbf D^{{}^{\geqslant {0}}}(Y).$$
\end{lem}

\begin{proof} First, we will show that
$$r^*[-c]\Z^\bullet_Y[n+1]\cong \Z^\bullet_{Y}[\dim Y]\in {}^{{}^\mu}\mathbf D^{{}^{\leqslant {0}}}(\Sigma),$$
which says that $r^*[-c]\Z^\bullet_Y[n+1]$ satisfies the support condition. This argument is simple.

Let $p$ be an integer. We need to show that $\dim \supp^{-p} \big(\Z^\bullet_{Y}[\dim Y]\big)\leq p$. We have 
$$\supp^{-p} \big(\Z^\bullet_{Y}[\dim Y]\big)=\overline{\{x\in Y \, | \, H^{\dim Y-p}( \Z^\bullet_{Y})_x\neq 0\}}.$$
Now $H^{\dim Y-p}( \Z^\bullet_{Y})_x=0$ unless $p=\dim Y$. Thus the support condition is satisfied.

\bigskip

Now we will show that  
$$r^![c]\Z^\bullet_\U[n+1]\in {}^{{}^\mu}\mathbf D^{{}^{\geqslant {0}}}(Y),$$
which says that $r^![c]\Z^\bullet_\U[n+1]$ satisfies the cosupport condition. Let $p$ be an integer. We need to show that $\dim \cosupp^{p} \big(r^![c]\Z^\bullet_\U[n+1]\big)\leq p$.

For all $x\in Y$, let $v_x:\{x\}\hookrightarrow Y$ denote the inclusion, and let $\hat v_x:=r\circ v_x$; hence, $\hat v_x$ is the inclusion of $\{x\}$ into $\U$. Then, we have 
$$\cosupp^{p} \big(r^![c]\Z^\bullet_\U[n+1]\big)=\overline{\{x\in Y \, | \, H^{p}(v_x^!r^![c]\Z^\bullet_\U[n+1])\neq 0\}}=
$$
\smallskip
$$\overline{\{x\in Y \, | \, H^{n+1+c+p}(\hat v_x^!\Z^\bullet_\U)\neq 0\}}=\overline{\{x\in Y \, | \, H^{n+1+c+p}(B^\circ_\epsilon(x), B^\circ_\epsilon(x)\backslash\{x\};\, \Z)\neq 0\}}=$$
\smallskip
$$
\overline{\{x\in Y \, | \, \widetilde H^{n+c+p}(S^{2n+1};\, \Z)\neq 0\}}.
$$

\medskip

\noindent where $B^\circ_\epsilon(x)$ again denotes a open ball of (small) radius $\epsilon>0$, centered at $x$, $S^{2n+1}$ denotes a sphere in $\U$ of real dimension $2n+1$ (its center and radius are irrelevant), and $\widetilde H$ denotes reduced cohomology. Now $\widetilde H^{n+c+p}(S^{2n+1};\, \Z)= 0$ unless $p=n+1-c=\dim Y$. Thus the cosupport condition is satisfied.
\end{proof}

\medskip

Now, as before, let $f:\U\rightarrow \C$ be a nowhere locally constant complex analytic function such that $V(f)$ is non-empty. Note that we have {\bf not} assumed that $f$ is reduced. Let $\Sigma_f:=\supp\phi_f[-1]\Z^\bullet_\U[n+1]$.

There is a notion of the singular set of the analytic set $V(f)$; it is the set of points at which $V(f)$ fails to be an analytic submanifold of $\U$ (which, using results about Milnor fibrations, is also the set of points where $V(f)$ fails to be even a topological submanifold of $\U$). We denote this singular set by $\Sigma V(f)$, and note that it always has dimension at most $n-1$. 

 If $f$ is not reduced, then $\Sigma_f$ will contain an irreducible component of $V(f)$, while $\Sigma V(f)$ will not. However, if $f$ is reduced, then $\Sigma_f=\Sigma V(f)$, and so the intermediate extension to $V(f)$ of the shifted constant sheaf on $V(f)\backslash\Sigma V(f)$ -- the intersection cohomology on $V(f)$ -- is the same as  the intermediate extension to $V(f)$ of the shifted constant sheaf on $V(f)\backslash\Sigma_f$, which is how we defined $\Idot_{V(f)}$ in \defref{defn:intermediate}.

\medskip

As in the general case, we let $j: V(f)\hookrightarrow \U$ and $\hat m:\Sigma_f\hookrightarrow\U$ denote the inclusions.

\medskip

The following lemma explains our terminology in \defref{def:pred}.

\begin{lem}\label{lem:classical}  The function $f$ is reduced (in the algebraic sense)  if and only if $f$ is $\Z^\bullet_\U[n+1]$-reduced, i.e., if and only if 
$$\hat m^*\Z^\bullet_\U[n+1]\in {}^{{}^\mu}\mathbf D^{{}^{\leqslant {-2}}}(\Sigma_f)\hskip 0.2in\textnormal{ and }\hskip 0.2in\hat m^!\Z^\bullet_\U[n+1]\in {}^{{}^\mu}\mathbf D^{{}^{\geqslant {2}}}(\Sigma_f).$$
\end{lem}
\begin{proof} 	Suppose that $$\hat m^*\Z^\bullet_\U[n+1]\in {}^{{}^\mu}\mathbf D^{{}^{\leqslant {-2}}}(\Sigma_f)\hskip 0.2in\textnormal{ and }\hskip 0.2in\hat m^!\Z^\bullet_\U[n+1]\in {}^{{}^\mu}\mathbf D^{{}^{\geqslant {2}}}(\Sigma_f).$$
Then, Item (4) of \thmref{thm:mainshort} tells us that $\Sigma_f$ does not contain an irreducible component of $V(f)$, i.e., $f$ is reduced.

\bigskip

Now we must prove the converse. Let $c$ be the codimension of $\Sigma f$ in $\U$, i.e., $c=n+1-\dim \Sigma_f$. Assume that $f$ is reduced, so that $\dim\Sigma_f\leq n-1$ and so $c\geq 2$. 

By \lemref{lem:fund}, we have that
$$\hat m^*[-c]\Z^\bullet_\U[n+1]\in {}^{{}^\mu}\mathbf D^{{}^{\leqslant {0}}}(\Sigma_f)\hskip 0.2in\textnormal{ and }\hskip 0.2in \hat m^![c]\Z^\bullet_\U[n+1]\in {}^{{}^\mu}\mathbf D^{{}^{\geqslant {0}}}(\Sigma_f),$$
that is 
$$\hat m^*\Z^\bullet_\U[n+1]\in {}^{{}^\mu}\mathbf D^{{}^{\leqslant {-c}}}(\Sigma_f)\hskip 0.2in\textnormal{ and }\hskip 0.2in \hat m^!\Z^\bullet_\U[n+1]\in {}^{{}^\mu}\mathbf D^{{}^{\geqslant {c}}}(\Sigma_f).$$
As $c\geq 2$, we are finished.
\medskip

\end{proof}

\medskip

From this lemma and \thmref{thm:mainshort}, we immediately conclude a new proof of ($\dagger$) and ($\ddagger$) from the introduction, which we state here as:

\smallskip

\begin{thm}\label{thm:class} Suppose that $f$ is reduced. Then, there are short exact sequences in the abelian category of perverse sheaves on $V(f)$:
$$
0\to \ker\{\operatorname{id}-\widetilde T_f\}\to \Z^\bullet_{V(f)}[n]\rightarrow \Idot_{V(f)}\to 0,
$$
and
$$
0\to \Idot_{V(f)}\to j^![1]\Z_\U^\bullet[n+1]\to\coker\{\operatorname{id}-\widetilde T_f\}\to 0,
$$

\medskip

\noindent where $\Idot_{V(f)}$ is the intersection cohomology complex on $V(f)$.
\end{thm}

\bigskip

As our final result, we will prove a theorem about integral cohomology (homology) manifolds. First, we need a lemma.

\medskip

\begin{lem}\label{lem:prehomman} For all $x\in V(f)$, let $E_{f,x}$ be the total space of the Milnor fibration of $f$ at $x$. Then the following are equivalent:

\begin{enumerate}
\item $j^![1]\Z_\U^\bullet[n+1]$ has stalk cohomology isomorphic to that of $j^*[-1]\Z_\U^\bullet[n+1]\cong \Z_{V(f)}^\bullet[n]$, and
\medskip
\item for all $x\in V(f)$, 
$$
H^{k}(E_{f,x}; \Z)\cong
 \begin{cases}
\Z, \ \textnormal{ if } k=1, 0;\\
0, \ \textnormal{ if } k\neq 1,0.
\end{cases}
$$
\end{enumerate}
\end{lem}
\begin{proof} Item (1) means that
$$
H^p(j^![1]\Z_\U^\bullet[n+1])_x \cong \begin{cases}
\Z, \ \textnormal{ if } p=-n;\\
0, \ \textnormal{ if } p\neq -n.
\end{cases}
$$
Now,
$$
H^p(j^![1]\Z_\U^\bullet[n+1])_x \cong \hyp^{n+p+2}(B_\epsilon^\circ(x),\, B_\epsilon^\circ(x)\backslash V(f); \Z) \cong \widetilde H^{n+p+1}(B_\epsilon^\circ(x)\backslash V(f); \Z),
$$
where $B_\epsilon^\circ(x)$ is a small open ball, centered at $x$, and $B_\epsilon^\circ(x)\backslash V(f)$ is homotopy-equivalent to $E_{f,x}$, the total space of the Milnor fibration of $f$ at $x$.

Thus, Item (1) is equivalent to
$$
\widetilde H^{k}(E_{f,x}; \Z)\cong
 \begin{cases}
\Z, \ \textnormal{ if } k=1;\\
0, \ \textnormal{ if } k\neq 1,
\end{cases}
$$
or, equivalently,
$$
H^{k}(E_{f,x}; \Z)\cong
 \begin{cases}
\Z, \ \textnormal{ if } k=1, 0;\\
0, \ \textnormal{ if } k\neq 1,0.
\end{cases}
$$
\end{proof}

\medskip

\begin{thm}\label{thm:pretau} Suppose that $f$ is reduced, and that $j^![1]\Z_\U^\bullet[n+1]$ has stalk cohomology isomorphic to that of $j^*[-1]\Z_\U^\bullet[n+1]$. Then,

\begin{enumerate}
\item $\id-\widetilde T_f$ is an isomorphism, 

\medskip

\item $j^*[-1]\Z_\U^\bullet[n+1]$ and $j^![1]\Z_\U^\bullet[n+1]$ are perverse sheaves,
\medskip
\item  $\omega_f:j^*[-1]\Z_\U^\bullet[n+1]\rightarrow j^![1]\Z_\U^\bullet[n+1]$ is an isomorphism, and both of these complexes are isomorphic to $\Idot_{V(f)}$,
\medskip
\item $V(f)$ is an integral cohomology/homology manifold, and
\medskip
\item there is an isomorphism 
$$
\psi_f[-1]\Z_\U^\bullet[n+1]\cong \phi_f[-1]\Z_\U^\bullet[n+1]\oplus \Idot_{V(f)}.
$$
\end{enumerate}
\end{thm}

\medskip

\begin{proof} 
The non-zero part of the cohomological version of the Wang sequence from Lemma 8.4 in \cite{milnorsing} begins as follows:

\bigskip

$$0\rightarrow H^0(E_{f,x};\, \Z)\rightarrow H^0(F_{f,x};\, \Z)\xrightarrow{(\operatorname{id}-T_f)^{-n}_x}H^0(F_{f,x};\, \Z)\rightarrow 
$$

$$H^1(E_{f,x};\, \Z)\rightarrow H^1(F_{f,x};\, \Z)\xrightarrow{(\operatorname{id}-T_f)^{-n+1}_x}H^1(F_{f,x};\, \Z)\rightarrow 
$$

$$H^2(E_{f,x};\, \Z)\rightarrow H^2(F_{f,x};\, \Z)\xrightarrow{(\operatorname{id}-T_f)^{-n+2}_x}H^2(F_{f,x};\, \Z)\rightarrow,$$

\medskip

\noindent where $F_{f,x}$ is the Milnor fiber of $f$ at $x$ , $T_f$ is the monodromy automorphism on the {\bf nearby} cycles, and the subscript $k$ in $(\operatorname{id}-T_f)^{k}_x$ denotes the degree (not exponentiation).

By \lemref{lem:prehomman}, we know that
$$
H^{k}(E_{f,x}; \Z)\cong
 \begin{cases}
\Z, \ \textnormal{ if } k=1, 0;\\
0, \ \textnormal{ if } p\neq 1,0.
\end{cases}
$$
Since $f$ is reduced, $F_{f,x}$ and $E_{f,x}$ are path-connected. We also know that $H^1(F_{f,x};\, \Z)$ is torsion-free by the Universal Coefficient Theorem. Therefore we obtain the exact sequence

$$0\rightarrow \Z\rightarrow \Z\xrightarrow{(\operatorname{id}-T_f)^{-n}_x}\Z\rightarrow\Z\rightarrow H^1(F_{f,x};\, \Z)\xrightarrow{(\operatorname{id}-T_f)^{-n+1}_x}H^1(F_{f,x};\, \Z)\rightarrow 
$$

$$0\rightarrow H^2(F_{f,x};\, \Z)\xrightarrow{(\operatorname{id}-T_f)^{-n+2}_x}H^2(F_{f,x};\, \Z)\rightarrow 0\dots,$$

\medskip

\noindent where the first map from $\Z$ to $\Z$ is an isomorphism, $(\operatorname{id}-T_f)^{-n}_x=0$, the third map from $\Z$ to $\Z$ is an isomorphism, and $(\operatorname{id}-T_f)^{k}_x$ is an isomorphism for $k\neq -n$. However, $T^k_{f,x}\cong \widetilde T^k_{f,x}$ outside of degree $-n$. Therefore, for all $x\in V(f)$, $(\operatorname{id}-\widetilde T_f)^{*}_x$ is an isomorphism and thus so is $\operatorname{id}-\widetilde T_f$.

\medskip

Now Items (2), (3), and (5) are immediate from \thmref{thm:split} and \thmref{thm:class}. It remains for us to demonstrate Item (4).

\medskip
For all $x\in V(f)$, let $v_x:\{x\}\hookrightarrow V(f)$ denote the inclusion. Let $\hat v_x:=j\circ v_x$ so that $\hat v_x$ is the inclusion of $\{x\}$ into $\U$. Applying $v_x^!$ to the isomorphism in Item (3), for all $x\in V(f)$, we have an isomorphism
$$
v_x^!j^*[-1]\Z_\U^\bullet[n+1]\cong v_x^!j^![1]\Z_\U^\bullet[n+1]\cong {\hat v}_x^![1]\Z_\U^\bullet[n+1]. \leqno{(\star)}
$$
Let $B^\circ_\epsilon(x)$ denote an open ball in $\U$ of (small) radius $\epsilon>0$, centered at $x$. Then ($\star$) tells us that, for all $p$, 
$$
\hyp^{p+n}(B^\circ_\epsilon(x)\cap V(f),\, B^\circ_\epsilon(x)\cap V(f)\backslash \{x\};\,\Z)\ \cong\ \hyp^{p+n+2}(B^\circ_\epsilon(x),\, B^\circ_\epsilon(x)\backslash \{x\};\,\Z),
$$
i.e., for all $k$,
$$
\hyp^{k}(B^\circ_\epsilon(x)\cap V(f),\, B^\circ_\epsilon(x)\cap V(f)\backslash \{x\};\,\Z)\ \cong\ \widetilde \hyp^{k+1}(S^{2n+1};\,\Z).
$$
As the real dimension of $V(f)$ is $2n$, we conclude that $V(f)$ is an integral cohomology/homology manifold.

\end{proof}

\medskip

\section{A question}

We continue with all of our previous notation.

\medskip

\thmref{thm:wang} and the definition of the intermediate extension motivate the following definition.

\medskip

\begin{defn} We define the {\bf intermediate Wang restriction} of $\Pdot$ to be the (perverse) image
$$
\Jdot_f:=\im\Big\{{}^{\mu}\hskip -0.02in H^0(j^*[-1]\Pdot)\xlongrightarrow{ {}^{\mu}\hskip -0.02in H^0(\omega_f)}{}^{\mu}\hskip -0.02in H^0(j^![1]\Pdot)\Big\}.
$$
\end{defn}

\medskip

The point is that now \thmref{thm:wang} and \thmref{thm:mainshort} tell us immediately that we have:

\medskip

\begin{prop} There are exact sequences in $\operatorname{Perv}(V(f))$:

$$
0\rightarrow\frac{\ker\{\id-\widetilde T_f\}}{\ker\{\operatorname{var}\}}\longrightarrow{}^{\mu}\hskip -0.02in H^0(j^*[-1]\Pdot)\xlongrightarrow{\alpha}\Jdot_f\rightarrow 0
$$
and 
$$
0\rightarrow \Jdot_f\xlongrightarrow{\beta} {}^{\mu}\hskip -0.02in H^0(j^![1]\Pdot)\longrightarrow\frac{\operatorname{im}\{\operatorname{can}\}}{\operatorname{im}\{\id-\widetilde T_f\}}\rightarrow 0,
$$
where $\beta\circ\alpha=  {}^{\mu}\hskip -0.02in H^0(\omega_f)$.

Furthermore, if $\hat m^*\Pdot\in {}^{{}^\mu}\mathbf D^{{}^{\leqslant {-2}}}(\Sigma)$ and $\hat m^!\Pdot\in {}^{{}^\mu}\mathbf D^{{}^{\geqslant {2}}}(\Sigma)$, then these two short exact sequences collapse to those of \thmref{thm:mainshort}; in particular, the intermediate Wang restriction $\Jdot_f$ is isomorphic to the intermediate extension $\Idot_{V(f)}$.
\end{prop}

\medskip

Of course, the big question is:

\medskip

\begin{ques}
Does $\Jdot_f$ have any interesting properties in general, even when it is not isomorphic to the intermediate extension $\Idot_{V(f)}$?
\end{ques}

\bibliographystyle{plain}

\bibliography{Masseybib}

\end{document}